\documentclass[11pt,reqno]{amsart} 

\usepackage{amsmath}
\usepackage{amsfonts,marvosym}
\usepackage{bbm}
\usepackage{amssymb}
\usepackage{amsthm}
\usepackage{caption}
\usepackage{fancyhdr}
\usepackage{hyperref}
\usepackage{enumitem}
\usepackage[margin=1in]{geometry}
\usepackage{graphicx}
\usepackage{float}
\usepackage{listings}
\usepackage{lmodern}
\usepackage[normalem]{ulem}
\usepackage{mathrsfs}
\usepackage{mathtools}
\usepackage{multicol}
\usepackage{xcolor}

\numberwithin{equation}{section}

\newtheorem{theorem}{Theorem}[section]

\newtheorem{lemma}[theorem]{Lemma}

\theoremstyle{definition}
\newtheorem{definition}[theorem]{Definition}

\def\Z{\mathbb{Z}}

\newcommand{\lcm}{\textnormal{lcm}}

\makeatletter
\def\subsection{\@startsection{subsection}{2}
  \z@{.5\linespacing\@plus.7\linespacing}
{.5\baselineskip}%
  {\normalfont\centering\scshape}%
}
\makeatother

\title{Dynamics of the Fibonacci Order of Appearance Map}

\author{Molly FitzGibbons}
\address{Department of Mathematics and Statistics, Williams College, Williamstown, MA 01267}
\email{\textcolor{blue}{\href{mailto:mcf4@williams.edu}{mcf4@williams.edu}}}

\author{Steven J. Miller}
\address{Department of Mathematics and Statistics, Williams College, Williamstown, MA 01267}
\email{\textcolor{blue}{\href{mailto:sjm1@williams.edu}{sjm1@williams.edu}}}

\author{Amanda Verga}
\address{Department of Mathematics, Trinity College, Hartford, CT 06106}
\email{\textcolor{blue}{\href{mailto:amanda.verga@trincoll.edu}{amanda.verga@trincoll.edu}}}

\begin{document}


\begin{abstract}
The \textit{order of appearance} $ z(n) $ of a positive integer $ n $ in the Fibonacci sequence is defined as the smallest positive integer $ j $ such that $ n $ divides the $ j $-th Fibonacci number. 
A \textit{fixed point} arises when, for a positive integer $ n $, we have that the $ n \textsuperscript{th} $  Fibonacci number is the smallest Fibonacci that $ n $ divides. 
In other words, $ z(n) = n $. 

In 2012, Marques proved that fixed points occur only when $ n $ is of the form $ 5^{k} $ or $ 12\cdot5^{k} $ for all non-negative integers $ k $. 
It immediately follows that there are infinitely many fixed points in the Fibonacci sequence. 
We prove that there are infinitely many integers that iterate to a fixed point in exactly $ k $ steps. 
In addition, we construct infinite families of integers that go to each fixed point of the form $12 \cdot 5^{k}$. 
We conclude by providing an alternate proof that all positive integers $n$ reach a fixed point after a finite number of iterations.  
\end{abstract}

\maketitle


\section{Introduction}

In $ 1202 $, the Italian mathematician Leonardo Fibonacci introduced the Fibonacci sequence $ \{F_{n}\}_{n=0}^{\infty} $, defined recursively as $ F_{n} = F_{n-1} + F_{n-2} $ with initial conditions $ F_{0} = 0 $ and $ F_{1} = 1 $.
By reducing $ \{F_{n}\}_{n=0}^{\infty} $ modulo $ m $, we obtain a periodic sequence $ \{F_{n} \mod m \}_{n=0}^{\infty} $.
This new sequence and its divisibility properties have been extensively studied, see for example \cite{M1,M5}. 
To see why the reduced sequence is periodic, note that by the pigeonhole principle if we look at $ n^{2} + 1 $ pairs $(F_{k}, F_{k-1})$, at least two are identical (mod $n$) and the recurrence relation generates the same future terms. 

\begin{definition}
\textnormal{The} \textit{order (or rank) of appearance} $ z(n) $ \textnormal{for a natural number $ n $ in the Fibonacci sequence is the smallest positive integer $ \ell $ such that $ n \mid F_{\ell} $.} 
\end{definition}

Observe that the function $ z(n) $ is well defined for all $ n $ since the Fibonacci sequences begins with $ 0, 1, \ldots $ and when reduced by modulo $ n $, a $0$ will appear again in the periodic sequence.
Thus, there will always be a Fibonacci number that is congruent to $ 0 \mod n $ for each choice of $ n $.
The upper bound of $n^2 + 1$ on $z(n)$ is improved in \cite{S}, which states $ z(n) \leq 2n $ for all $ n \geq 1 $.
This is the sharpest upper bound on $z(n)$. 
In \cite{M2}, sharper upper bounds for $ z(n) $ are provided for some positive integers $ n $.
Additional results on $ z(n) $ include explicit formulae for the order of appearance of some $ n $ relating to sums containing Fibonacci numbers \cite{M3} and products of Fibonacci numbers \cite{M4}.
We study repeated applications of $ z $ on $ n $ and denote the $ k \textsuperscript{th} $ application of $ z $ on $ n $ as $ z^{k}(n) $. 
We are interested in the following quantity.

\begin{definition}
\textnormal{The} \textit{fixed point order} \textnormal{for a natural number $ n $ is the smallest positive integer $k$ such that $z^{k}(n)$ is a fixed point. If $n$ is a fixed point, then we say the \textit{fixed point order} of $n$ is 0.}
\end{definition}

Table \ref{table:zkiterations} shows which values occur after repeated iterations of $ z $ on the first $12$ positive integers. 
We further the study of repeated iterations of $ z $ on $ n $. 





In Section \ref{auxresults}, we provide some useful properties of the order of appearance in Fibonacci numbers.
In the remaining sections, we prove our main results, found below.

\begin{theorem}\label{infntakek} 
For all positive integers $ k $, there exist infinitely many $ n $ with fixed point order $ k $.
\end{theorem}

\begin{theorem} \label{infngotoeachFP}
Infinitely many integers $n$ iterate to each fixed point of the form $12 \cdot 5^{k}$. 
\end{theorem}

\begin{theorem}\label{allngotofp}
All positive integers $n$ have finite fixed point order. 
\end{theorem}
Theorem \ref{allngotofp} was proved first in \cite{LT} by showing that within finite $k$, $z^k(n) = 2^a3^b5^c$ where $a,b,c \in \Z_{\geq 0}$ and then proving that $2^a3^b5^c$ iterates to a fixed point in a finite number of steps. It was later proved in \cite{Ta1} using a relationship between the Pisano period of $n$ and $z(n)$. We provide an alternate proof using a minimal counterexample argument. 

\begin{table}[htb]
\centering
\begin{tabular}{|c|c|c|c|c|c|} \hline
$ n \setminus k $ & 1  & 2  & 3  & 4  & 5   \\ \hline
1        & \textbf{1}  &    &    &    &     \\ 
2        & 3  & 4  & 6  & \textbf{12} &     \\ 
3        & 4  & 6  & \textbf{12} &    &     \\ 
4        & 6  & \textbf{12} &    &    &     \\ 
5        & 5  &    &    &    &     \\ 
6        & \textbf{12} &    &    &    &     \\ 
7        & 8  & 6  & \textbf{12} &    &     \\ 
8        & 6  & \textbf{12} &    &    &     \\ 
9        & \textbf{12} &    &    &    &     \\ 
10       & 15 & 20 & 30 & \textbf{60} &     \\ 
11       & 10 & 15 & 20 & 30 & \textbf{60}  \\ 
12       & \textbf{12} &    &    &    &     \\ \hline
\end{tabular}
\caption{Iterations of $ z $ on $ n $, numbers in bold are fixed points.}
\label{table:zkiterations}
\end{table} 


\section{Auxiliary Results} \label{auxresults}

Here we include some needed results from previous papers. 

\begin{lemma}\label{allfixedpointsofform}
Let $ n $ be a positive integer. 
Then $ z(n) = n $ if and only if $ n = 5^{k}$ or $ n = 12 \cdot 5^{k} $ for some $ k \geq 0 $.
\end{lemma}

A proof of Lemma \ref{allfixedpointsofform} can be found in \cite{M1, SM}. 

\begin{lemma}\label{pow2equals}
For all $ a \in \Z, a \geq 3 $, $ z\left(2^{a}\right) = 2^{a-2} \cdot 3 $. For all $b \in \Z^{+}, z(3^b) = 4 \cdot 3^{b-1}$.
\end{lemma}

Lemma \ref{pow2equals} is Theorem 1.1 of \cite{M2}. 

\begin{lemma}\label{z(n)=lcm(p^e's)}
Let $ n \geq 2 $ be an integer with prime factorization $ n = p_{1}^{e_{1}}p_{2}^{e_{2}} \cdots p_{m}^{e_{m}} $ where $ p_{1}, p_{2}, \ldots, p_{m} $ are prime and $e_{1}, e_{2}, \ldots, e_{m} $ are positive integers. 
Then 
\begin{equation}
z(n)  \ = \  z(p_{1}^{e_{1}}p_{2}^{e_{2}}\cdots p_{m}^{e_{m}})  \ = \  \lcm\left( z(p_{1}^{e_{1}}), z(p_{2}^{e_{2}}), \ldots, z(p_{m}^{e_{m}})\right).
\end{equation}
\end{lemma}

A proof of Lemma \ref{z(n)=lcm(p^e's)} can be found in Theorem $ 3.3 $ of \cite{R}. 
Lemma \ref{z(n)=lcm(p^e's)} has been generalized as follows.

\begin{lemma}\label{general z(n)=lcm(p^e's)}
Let $ n \geq 2 $ be an integer with prime factorization $ n = p_{1}^{e_{1}}p_{2}^{e_{2}} \cdots p_{m}^{e_{m}} $ where $ p_{1}, p_{2}, \ldots, p_{m} $ are prime and $e_{1}, e_{2}, \ldots, e_{m} $ are positive integers. 
Then 
\begin{equation}
z \left(\lcm( m_{1},m_{2},\ldots, m_{n}) \right) 
 \ = \  \lcm \left(z(m_{1}), z(m_{2}), \ldots, z(m_{n})\right).
\end{equation}
\end{lemma}
A proof of Lemma \ref{general z(n)=lcm(p^e's)} can be found in Lemma $ 4 $ of \cite{Ty}.

\begin{lemma}\label{z(p)leqp+1}
   For all primes $p$, $z(p)\leq p+1$. 
\end{lemma}
A proof of Lemma \ref{z(p)leqp+1} can be found in Lemma 2.3 of \cite{M1}. 

\begin{lemma}\label{z(n)leq2n}
For all positive integers $n$, $z(n) \leq 2n$, with equality if and only if $n=6\cdot 5^k$ for some $k \in \Z_{\geq 0}$
\end{lemma}
Lemma \ref{z(n)leq2n} is proven in \cite{S}. 


\begin{lemma}\label{z(p)relprimetop}
For all primes $ p \neq 5 $, we have that $ \gcd\left(p, z(p)\right) =  1 $.
\end{lemma}

Lemma \ref{z(p)relprimetop} is proven in Lemma 2.3 of \cite{M1}.

\begin{lemma} \label{ifndividesF_m}
If $n \vert F_m$, where $F_m$ is the $m\textsuperscript{th}$ number in the Fibonacci sequence, then $z(n) \vert m$.
\end{lemma}
Lemma \ref{ifndividesF_m} is Lemma 2.2 of \cite{M1}. 

\begin{lemma}\label{z(p^e)=z(p^r)z(p)}
For all odd primes $p$, we have $ z\left(p^{e}\right) = p^{\max (e - a, 0)}z(p)$ where $a$ is the number of times that $p$ divides $F_{z(p)}$, $a \geq 1$.
In particular, $ z\left(p^{e}\right) = p^{r}z(p) $ for some $0 \leq r \leq e-1 $.
\end{lemma}
For a proof of Lemma \ref{z(p^e)=z(p^r)z(p)}, see Theorem 2.4 of \cite{FM}.



\section{Infinitely many integers take a given number of iterations to reach a fixed point} \label{infnreachfpink}

In this section we first prove Lemma \ref{coefconst}, which helps us show that when $z^{i}(n)$ is written as the product of a constant relatively prime to $ 5 $ and a power of $ 5 $, then $z^{i}\left(n \cdot 5^{a}\right)$ can be written as the product of that same constant and another power of $ 5 $. 
Table \ref{table2}
lists the smallest $ n $ that takes exactly $ k $ iterations to reach a fixed point for positive integers $ k $ up to 10. 
\begin{table}[htb]
\centering
\begin{tabular}{|c||c|c|} \hline
k    & n & FP  \\ \hline
1    & 1 & 1 \\
2    & 4 & 12 \\
3    & 3 & 12 \\
4    & 2 & 12  \\
5   & 11 & 60  \\
6   & 89 & 60  \\
7 & 1069 & 60  \\
8 & 2137 & 60  \\
9 & 4273 & 60  \\ 
10 & 59833 & 60 \\ \hline
\end{tabular}
\caption{First $ n $ that takes $ k $ iterations to reach a fixed point.}
\label{table2}
\end{table}

\begin{lemma}\label{coefconst}
Let $z^{i}(5^{a}\cdot n) = c_{(i,a, n)}5^{a_i}$, where $c_{(i, a, n)}$ is a constant that is relatively prime to $5$ and depends on $i$ and $n$, and $a_i \in \Z^+$. 
Fix $i, n \in \Z_{\geq 0}$. 
Then $c_{(i, a, n)}$ remains the same for all choices of $ a $.
\end{lemma}

\begin{proof}
Let the prime factorization of an integer $ n $ be $ n = 5^{e_1}p_{2}^{e_2} \cdots p_{m}^{e_m}$ where $e_{1} \geq 0$ and each of the  $e_{2},e_{3},\ldots,e_{m} \geq 1$.

We proceed by induction on the number of iterations of $ z $.
First suppose $i = 1$ and let the prime factorization of 
$\lcm \left( z(p_{2}^{e_2}), \ldots, z(p_{m}^{e_m}) \right) $ equal $ 5^{f_1}q_{2}^{f_2} \cdots q_{r}^{f_r}$ where $f_1 \geq 0$ and each of the $f_2,\ldots,f_r \geq 1$. Observe that

\begin{align}
z(5^{a}\cdot n)  
&  \ = \  \lcm\left(z(5^{e_1 +a}), z(p_{2}^{e_2}), \ldots, z(p_{m}^{e_m})\right) \nonumber\\ 
&  \ = \  \lcm\left(z\left(5^{e_1 +a}\right), \lcm \left(z(p_{2}^{e_2}), \ldots, z(p_{m}^{e_m}) \right) \right) \nonumber \\
&  \ = \  \lcm\left(5^{e_1+a}, 5^{f_1}q_{2}^{f_2}\cdots q_{r}^{f_r}\right) \nonumber\\
&  \ = \  q_{2}^{f_2} \cdots q_{r}^{f_r} \cdot 5^{\text{max}(e_1+a, f_1)}.
\end{align}
Thus, $c_{(1, a, n)} = q_{2}^{f_2} \cdots q_{r}^{f_{r}}$ for any non-negative integer $ a $ when $ n $ is not a power of $5$ or when $\lcm\left( z(p_2^{e_2}), \ldots, z(p_m^{e_m}) \right)$ is not a power of $5$, and $c_{(1, a, n)} = 1$ otherwise.

Next, assume that for some $ i, n \in \mathbb{Z}^{+} $, we have $z^{i}(5^{a}\cdot n) = c_{(i, a, n)}5^{a_i} $ where $c_{(i,a,n)}$ is the same for all choices of $a \in \mathbb{Z}_{\geq 0}$. 
First suppose $c_{(i,a, n)}=1$. 
Then
\begin{align}
z^{i+1}(5^a \cdot n)
&  \ = \  z(z^{i}(5^{a}\cdot n)) \nonumber\\
&  \ = \  z\left((c_{(i,a,n)}) \cdot 5^{a_i} \right) \quad\quad \text{where $ a_i \in \Z_{\geq0} $} \nonumber\\
&  \ = \  \lcm\left(z(1), z(5^{a_i})\right) \nonumber\\
& \ = \ 5^{a_i}. 
\end{align}
Therefore, for any choice of $a$, $c_{(i+1,a,n)}=1$.

Now suppose $c_{(i,a,n)} \neq 1$. Then let the prime factorization of 
$c_{(i,a,n)} = q_{1}^{f_{1}} \cdots q_{r}^{f_{r}}$, where 
$q_{1},\ldots,q_{r} \neq 5 $ since $ \gcd \left(c_{(i,a,n)},5\right) = 1 $. 
Let 
$ \lcm\left( z(q_{1}^{f_1}),\ldots,z(q_{r}^{f_r})\right) =  5^{g_1}h_{2}^{g_2}\cdots h_{j}^{g_j}$ 
where $h_{2}, \ldots, h_{j}$ are primes not equal to 5. 
Then
\begin{align}
z^{i+1}(5^a \cdot n)
&  \ = \  z(z^{i}(5^{a}\cdot n)) \nonumber\\
&  \ = \  z\left((c_{(i,a,n)}) \cdot 5^{a_i} \right) \quad\quad \text{where $ a_i \in \Z_{\geq0} $} \nonumber\\
&  \ = \  \lcm\left( z(q_{1}^{f_1}),\ldots,z(q_{r}^{f_r}), z(5^{a_i})\right) \nonumber\\
&  \ = \  \lcm\left( \lcm\left(z(q_{1}^{f_1}),\ldots,z(q_{r}^{f_r})\right), z(5^{a_i})\right) \nonumber\\
&  \ = \  \lcm\left(5^{g_1}h_{2}^{g_2} \cdots h_{j}^{g_j}, 5^{a_i}\right) \nonumber\\
&  \ = \  h_{2}^{g_2} \cdots h_{j}^{g_j} \cdot 5^{\text{max}(g_1, t)} \\
&  \ = \ c\left(i,a, n\right) \cdot 5^{\text{max}\left(g,t\right)}.
\end{align} 
\end{proof}

We use Lemma \ref{coefconst} in our proof of Theorem \ref{infntakek} to show that if there exists an integer $n$ that takes exactly $k$ iterations of $z$ to reach a fixed point, then there are infinitely many integers that take exactly $k$ iterations of $z$ to reach a fixed point. 
The following lemma provides us with information on the $k\textsuperscript{th}$ iteration of $z$ on powers of $10$, enabling us to find integers that require exactly $k$ iterations of $z$ to reach a fixed point for any positve integer $k$.

\begin{lemma}\label{10mgoestoink}
For all $k, m \in \mathbb{Z}, k\geq 0, m\geq 4$ and $2k+2\leq m$, $z^{k}(10^{m}) = 3\cdot5^m\cdot2^{m-2k}$.
\end{lemma}

\begin{proof}
We proceed by induction on the number of iterations of $ z $.
Observe that when $k = 1$,
\begin{align}
z\left(10^{m}\right) 
&  \ = \  \lcm \left(z\left(2^{m}\right),z\left(5^{m}\right)\right) \nonumber\\
&  \ = \  \lcm \left(3\cdot 2^{m-2}, 5^m\right) \nonumber\\
&  \ = \  3 \cdot5^{m} \cdot 2^{m-2}.
\end{align}
Now suppose that $z^{k}(10^{m}) \ = \ 3\cdot5^m\cdot2^{m-2k}$ for some positive integer $k$. 
Then we have
\begin{align}
z^{k+1}(10^{m})
& \ = \ z\left(z^{k}\left(10^{m}\right)\right) \nonumber\\
&  \ = \  z\left( 3\cdot5^m\cdot2^{m-2k}\right) \nonumber\\
&  \ = \  \lcm \left(z\left(3\right),z\left(5^{m}\right),z(2^{m-2k}) \right) \nonumber\\
&  \ = \  \lcm \left(4, 5^m,2^{m-2k-2} \cdot 3\right). 
\end{align}  
By assumption, $ m \geq 2(k+1)+2  =  2k+4 $, thus we have $z^{k+1}(10^{m}) \ = \  3 \cdot 5^{m} \cdot 2^{m-2(k+1)}$.
\end{proof}

Using Lemmas \ref{coefconst} and \ref{10mgoestoink}, we now prove Theorem \ref{infntakek}: \\

\textit{For all $k \in \Z_{\geq 0}$, there exist infinitely many $n$ with fixed point order $k$.}

\begin{proof} 
[Proof of Theorem \ref{infntakek}]
Let $g, h \in \Z^{+}, g>h$. Then $g=h+\ell$ for some $\ell \in \Z^{+}$. 
Suppose that $z^h(n)$ is a fixed point. 
Then $z^g(n) = z^{\ell}\left(z^{h}(n) \right)$, so $z^{g}(n)$ is also a fixed point. 
Similarly, if $z^{g}(n)$ is not a fixed point, then $z^{h}(n)$ cannot be a fixed point for any $h<g$. 

Note that by Lemma \ref{10mgoestoink}
\begin{equation}
z^{k}(10^{2k+2})  \ = \  3 \cdot 5^{2k+2} \cdot 2^{(2k+2)-2k}  \ = \  12 \cdot 5^{2k+2}
\end{equation}
and
\begin{equation}
z^{k-1}(10^{2k+2})  \ = \  3 \cdot 5^{2k+2} \cdot 2^{(2k+2)-2(k-1)}  \ = \  12 \cdot 5^{2k+2} \cdot 2^{2}.
\end{equation}
Thus, $10^{2k+2}$ takes exactly $k$ iterations of $z$ to reach a fixed point, as $z^{k-1}(10^{2k+2})$ is not a fixed point. 
We prove that we can find infinitely many integers that take exactly $k$ iterations to reach a fixed point once one such integer is identified (which we have just done). 

We first consider the case where an integer $ n $ goes to a fixed point of the form $ 12 \cdot 5^{a'} $, where $a' \in \mathbb{Z}_{\geq 0}$, in exactly $k$ iterations of $z$. 
Thus, $ z^{k}(n) =  12\cdot5^{a'} $ and $ z^{k-1}(n) = c\cdot5^{b'} $ for some non-negative integer $ b' $ and positive integer $ c \neq 1,12 $.
Let $ r $ be an arbitrary positive integer.
By Lemma \ref{coefconst}, we have $ z^{k}\left(5^{r}\cdot n\right) =  12 \cdot 5^{a''} $ and $ z^{k-1}\left(5^{r}\cdot n\right) = c \cdot 5^{b''} $ for non-negative integers $ a'' $ and $ b'' $.
Thus, $ 5^{r}\cdot n $ requires exactly $ k $ iterations to reach a fixed point.
Next we consider the case where $ n $ goes to a fixed point of the form $ 5^{a'} $ in exactly $ k $ steps.
Then, $ z^{k}(n) =  5^{a'} $ and $ z^{k-1}(n) = c \cdot 5^{b'} $ for some non-negative integer $ b' $ and positive integer $ c \neq 1,12 $.
Let $ r $ be an arbitrary positive integer.
By Lemma \ref{coefconst}, we have $ z^{k}\left(5^{r}\cdot n\right) = 5^{a''} $ and $ z^{k-1}\left(5^{r}\cdot n\right) = c \cdot 5^{b''} $ for non-negative integers $ a'' $ and $ b'' $.
Thus, $ 5^{r}\cdot n $ requires exactly $ k $ iterations to reach a fixed point. 
As $ r $ is arbitrary, there are infinitely many integers with fixed point order $ k $ for any positive integer $ k $.
\end{proof}


\section{Infinitely many integers go to each fixed point}

We begin this section with a proof about the $ k \textsuperscript{th} $ iteration of $z$ on powers of $2$. 

\begin{lemma} \label{z^i(2^a)equals}
For all $k,a \in \mathbb{Z}$ such that $2\leq k$ and $4\leq a$, $z^{k}(2^{a}) = \lcm \left( 2^{a-2k} \cdot 3, 4 \right)$.
\end{lemma}

\begin{proof} We induct on $k$; we use Lemma \ref{pow2equals} to note that $z(2^{a})=z^{a-2}\cdot 3$ (valid as $a \geq 3$) with base case $k=2$:
\begin{align}
z^{2}(2^{a}) 
&  \ = \  z\left(z(2^{a})\right) \nonumber\\
&  \ = \  z(2^{a-2} \cdot 3) \nonumber\\
&  \ = \  \lcm\left( z(2^{a-2}), z(3) \right) \nonumber\\
&  \ = \  \lcm\left(2^{a-4}\cdot 3, 4 \right).
\end{align}
For the inductive step, assume that $z^{k}(2^{a}) = \lcm \left( 2^{a-2k} \cdot 3, 4 \right)$ for some $k$. 
We show that $z^{k+1}(2^{a})  \ = \  \lcm \left( 2^{a-2(k+1)} \cdot 3, 4 \right)$. 
First suppose that $a > 2k+2$. Then
\begin{align}
z^{k+1}(2^{a}) 
&  \ = \   z\left(z^{k}(2^{a})\right) \nonumber\\
&  \ = \  z\left(\lcm \left( 2^{a-2k} \cdot 3, 4 \right)\right) \nonumber\\
&  \ = \  z\left(2^{a-2k} \cdot 3\right)\nonumber \\
&  \ = \  \lcm \left(z(2^{a-2k}), z(3)\right)\nonumber \\
&  \ = \  \lcm \left( 2^{a-2k-2} \cdot 3, 4 \right) \nonumber\\
&  \ = \  \lcm \left( 2^{a-2(k+1)} \cdot 3, 4 \right).
\end{align}
Now suppose that $a \leq 2k+2$. 
Then 
\begin{align}
z^{k+1}(2^{a}) 
&  \ = \  z\left(z^{k}\left(2^{a}\right)\right) \nonumber\\
&  \ = \  z\left(\lcm \left( 2^{a-2k} \cdot 3, 4 \right)\right) \nonumber\\
&  \ = \  z\left(12\right) \nonumber\\
&  \ = \  12 \nonumber\\
&  \ = \  \lcm \left( 2^{a-2(k+1)} \cdot 3, 4 \right).
\end{align}
\end{proof}

We now use Lemma \ref{z^i(2^a)equals} in our proof of Lemma \ref{all2to12}, which proves that all powers of 2 go to the fixed point $12$ and determines how many iterations of $z$ it takes for a power of 2 to reach $12$.

\begin{lemma} \label{all2to12}
For all $a \in \Z^{+}$, $2^{a}$ reaches the fixed point $12$ in finitely many iterations of $z$. 
For $a \geq 4$, exactly $\lceil \frac{a}{2}\rceil -1 $  iterations of $z$ are required to reach $12$.
\end{lemma}

\begin{proof} 
When $a\leq 4$, the claim follows from straightforward computation. 
Notice that $z^{4}(2) = 12, z^{2}(2^{2}) = 12, z^{2}(2^{3}) = 12, z(2^{4}) = 12$.
We prove for $a>4$ using Lemma \ref{z^i(2^a)equals}.

Note that if $ a $ is even, then $\lceil\frac{a}{2}\rceil = \frac{a}{2}$.
Thus, in the case where $a$ is even, 
\begin{equation}
z^{\lceil\frac{a}{2}\rceil-1}(2^{a})  \ = \  \lcm(2^{a-2(\frac{a}{2}-1)} \cdot 3, 4)  \ = \  \lcm (2^{a-a+2} \cdot 3,4)  \ = \  \lcm (2^{2} \cdot 3,4)  \ = \  12.
\end{equation}
So $2^a$ takes at most $ \left\lceil \frac{a}{2} \right\rceil -1$ iterations of $z$ to reach a fixed point when $ a $ is even. 
We next show that it takes exactly $\lceil \frac{a}{2}\rceil -1$ by showing that $z^{(\lceil \frac{a}{2}\rceil-1)-1}(2^a)$ is not a fixed point:
\begin{equation}
z^{\left(\left\lceil\frac{a}{2}\right\rceil-1\right)-1}(2^{a}) =  \lcm(2^{a-2(\frac{a}{2}-2)} \cdot 3, 4) \ = \ \lcm (2^{a-a+4} \cdot 3,4) = \lcm (2^{4} \cdot 3,4) = 12\cdot2^{2},
\end{equation}
which is not a fixed point.
When $a$ is odd, $ \left\lceil \frac{a}{2} \right\rceil -1 = \frac{a-1}{2}$, giving us
\begin{equation}
z^{\left\lceil \frac{a}{2} \right\rceil-1}(2^{a})  \ = \  \lcm(2^{a-2(\frac{a-1}{2})} \cdot 3, 4)  \ = \  \lcm (2^{a-a+1} \cdot 3,4)  \ = \  \lcm (2 \cdot 3,4)  \ = \  12.
\end{equation}
However
\begin{equation}
z^{(\lceil \frac{a}{2}\rceil-1)-1}(2^{a})  \ = \  \lcm(2^{a-2(\frac{a-1}{2}-1)} \cdot 3, 4)  \ = \  \lcm (2^{a-a+1+2} \cdot 3,4)  \ = \  \lcm (2^{3} \cdot 3,4)  \ = \  12\cdot2,
\end{equation}
which is not a fixed point.
\end{proof}

Lemma \ref{all2to12} and Lemma \ref{coefconst} now yield Theorem \ref{infngotoeachFP}:

\medskip
\medskip

\textit{Infinitely many integers $n$ go to each fixed point of the form $12 \cdot 5^{k}$. }

\begin{proof}[Proof of Theorem \ref{infngotoeachFP}]
Using Lemma \ref{all2to12}, we know that $ z^{\lceil\frac{a}{2}\rceil-1}(2^{a} \cdot 5^{0}) = 12 $. 
Thus, by Lemma \ref{coefconst}, $z^{\lceil\frac{a}{2}\rceil-1}(2^{a} \cdot 5^{b}) = 12 \cdot 5^{b'}$ for some nonnegative integer $b'$.
We show that $ b = b' $ by inducting on $ t $ to show that $ z^{t}(2^a \cdot 5^{b}) =  2^{a'} \cdot 3 \cdot 5^{b}, a'\in\mathbb{Z}^{+}$, for all $ a > t, a > 2 $. 
When $ t = 1 $,
\begin{align}
z(2^{a}\cdot 5^{b}) 
&  \ = \  \lcm\left( z(2^{a}), z(5^{b})\right) \nonumber\\
&  \ = \  \lcm\left( 2^{a-2} \cdot 3, 5^{b} \right) \nonumber\\
&  \ = \   2^{a-2} \cdot 3 \cdot 5^{b}.
\end{align}
Now suppose that $z^{t}(2^a \cdot 5^{b}) = 2^{a'} \cdot 3 \cdot 5^{b}$ for some positive integer $a'$. Then
\begin{align}
z^{t+1}(2^{a}\cdot 5^{b}) 
&  \ = \   z\left(z^{t}(2^{a}\cdot 5^{b})\right) \nonumber\\
&  \ = \  z\left( 2^{a'} \cdot 3 \cdot 5^{b} \right) \nonumber\\
&  \ = \  \lcm\left( z(2^{a'}), z(3), z(5^{b})\right).
\end{align}
If $a' \leq 3 $, then $\lcm\left( z(2^{a'}), z(3), z(5^{b})\right) = 2^{2} \cdot 3 \cdot 5^{b}$. 
If $a' > 3$, then
\begin{align}
\lcm\left( z(2^{a'}), z(3), z(5^{b})\right) 
&  \ = \  \lcm\left( 2^{a'-2} \cdot 3, 4, 5^{b} \right) \nonumber\\
&  \ = \   2^{a'-2} \cdot 3 \cdot 5^{b}. 
\end{align}
A straightforward calculation shows that $2 \cdot 5^{b}$ and $2^{2} \cdot 5^{b}$ iterate to the fixed point $12 \cdot 5^{b}$ (see Appendix \ref{2,4timespow5iterateto12} for a proof).
Therefore $2^{a}\cdot 5^{b}$ iterates to the fixed point $12 \cdot 5^{b}$ for all $a \in \mathbb{Z}^{+}$.
\end{proof}


\section{All integers have finite fixed point order}

We now prove that when $a,b$ are relatively prime, $z^k(ab) = \lcm (z^k(a),z^k(b))$. 
We will use this in the proof of Theorem \ref{allngotofp}.
\begin{lemma}\label{z^k(n)=lcm(z^k(a),z^k(b))}
Let $n = ab$ where $\gcd (a,b) = 1$. 
Then $z^{k}(n) = \lcm(z^{k}(a),z^{k}(b))$.
\end{lemma}

\begin{proof}
We first consider the case where $n$ has only one prime in its prime factorization.
Without loss of generality, suppose $ a = 1 $ and $ b = n $ and $z^{k}(n) = \lcm(1, z^{k}(n))$. 
If $ n = 1 $, then $ a = b = 1 $ and $z^k\left(1\right) = 1 = \lcm\left(z^k(1), z^k(1)\right)$.

Next consider when $n$ has at least two distinct primes in its prime factorization. 
Let the prime factorization of $n = p_{1}^{e_1}\cdots p_{m}^{e_m}$ and let $a = p_{1}^{e_1}\cdots p_{r}^{e_r}$, $b = p_{r+1}^{e_{r+1}}\cdots p_{m}^{e_m}$ where $1 \leq r < m$. Note that the primes are not necessarily in increasing order. 
We proceed by induction. 
In the base case $k = 1$, and using Lemma \ref{z(n)=lcm(p^e's)} we have:
\begin{align}
z(n)  
& \ = \  \lcm \left( z \left(p_{1}^{e_1}\right), \ldots, z(p_{r}^{e_r}), z(p_{r+1}^{e_{r+1}}),\ldots, z(p_{m}^{e_m}) \right)  \nonumber \\ 
& \ = \  \lcm \left( \lcm \left(z(p_{1}^{e_1}), \ldots, z(p_{r}^{e_r}) \right), \lcm \left(z\left(p_{r+1}^{e_{r+1}}\right),\ldots, z(p_{m}^{e_m}) \right) \right) \nonumber \\
& \ = \  \lcm \left( z(p_{1}^{e_1}\cdots p_{r}^{e_r}), z(p_{r+1}^{e_{r+1}}\cdots p_{m}^{e_m}) \right) \nonumber \\ 
&\ = \  \lcm \left( z(a), z(b) \right).
\end{align}

For the inductive step, assume that for some $k \geq 1$, $z^{k}(n) = \lcm\left(z^{k}(a),z^{k}(b)\right)$. 
We show that $z^{k+1}(n) = \lcm\left(z^{k+1}(a),z^{k+1}(b)\right)$.
We have
\begin{align}
z^{k+1}(n)  
& \ = \  z\left(z^k(n) \right) \nonumber \\
& \ = \  z\left(\lcm(z^{k}(a),z^{k}(b)) \right) \nonumber \\ 
&  \ = \ \lcm\left(z(z^{k}(a)),z(z^{k}(b)) \right) \quad\quad \text {by Lemma 2.4} \nonumber\\ 
&  \ = \   \lcm\left( z^{k+1}(a),z^{k+1}(b) \right).
\end{align}
\end{proof}

Now we are ready to prove Theorem \ref{allngotofp}. 

\begin{proof}
[Proof of Theorem \ref{allngotofp}]
Suppose that $n$ is the smallest positive integer with undefined fixed point order. 
We prove first the case where $n=ab$ where $\gcd(a,b)=1$ and $a,b \geq 2$, then prove cases where $n$ is a power of a prime. 
\medskip \newline
\textit{Case 1:}  
First suppose that $n$ has at least two distinct primes in its prime factorization, so $n$ can be written $n = ab$ where $\gcd(a,b) = 1$ and $ a,b >1$. 
Since $a,b < n$, we know that $a,b$ have finite fixed point order.
Suppose the fixed point order of $a$ is $c$ and the fixed point order of $b$ is $d$.  
Let $k = \text{max}(c,d)$. 
Then by Lemma \ref{z^k(n)=lcm(z^k(a),z^k(b))} we have $z^k(n) = \lcm\left(z^k(a), z^k(b) \right)$ telling us $z^k(n)$ is a fixed point, which is a contradiction. 
\medskip \newline
\textit{Case 2:} Now suppose that $n = p$ for some prime $p$. Notice that $p \neq 2$ since we prove in Lemma \ref{z^i(2^a)equals} that powers of $2$ reach the fixed point 12 in finitely many iterations of $z$. 
By Lemma \ref{z(p)leqp+1}, $z(p)\leq p+1$. 
Note that $z(p) \neq p$, or else $p$ would have fixed point order of $0$. 
Thus, $z(p) = p+1$ or $z(p) < p$

Since we are assuming $p$ does not iterate to a fixed point, neither does $z(p)$. 
Thus $z(p)$ is not a power of $2$ since powers of $2$ iterate to a fixed point by Lemma \ref{z^i(2^a)equals}.
Thus if $z(p) = p+1$ then $z(p) = 2^r \cdot t$ where $r \in \Z^+$ (since $p+1$ is even) and $t \in \Z, t \geq 3$, $ \gcd(2,t)=1$ and $2^r, t < p$. So by the same argument as in Case 1, $z(p)$ has finite fixed point order, so $p$ also has finite fixed point order since it reaches a fixed point after one more iteration of $z$ than $z(p)$. 

If $z(p) < p$ then $z(p)$ has finite fixed point order by the assumption that $p$ is the smallest integer with undefined fixed point order. Thus $p$ also has finite fixed point order. 
\medskip \newline
\textit{Case 3:} Now suppose $n = p^e$ where $p$ is prime and $e \geq 2$. From Lemmas \ref{z(p^e)=z(p^r)z(p)} and \ref{z(n)leq2n} we know that for some $r \in \Z_{\geq 0}, r<e$, we have $ z(p^e) = p^r z(p)$. As $e>1$, we have $z(p) \leq p+1 < p^e$ by Lemma \ref{z(p)leqp+1}. Thus, $z(p)$ has finite fixed point order. Notice that $p^r<p^e$, so $p^r$ also has finite fixed point order. Let $h = \max (\text{fixed point order of } z(p), \text{fixed point order of } p^r)$. 


Note that $\gcd (z(p),p^r)=1$ since $z(p)$ is relatively prime to $p$ by Lemma \ref{z(p)relprimetop}. Then by Lemma \ref{z^k(n)=lcm(z^k(a),z^k(b))},
\begin{equation}
z^{h+1}(p^e) 
\ = \ z^{h}\left(z\left(p^{e}\right)\right) 
\ = \ z^h \left( p^r z(p)\right)
\ = \ \lcm\left( z^h(p^r), z^h(z(p)) \right).
\end{equation}
Thus, $p^e$ iterates to a fixed point within $h+1$ iterations of $z$, so $p^e$ has finite fixed point order. 
\end{proof}


\section{Acknowledgements}
The authors were supported by the National Science Foundation under Grants No. DMS-2241623 and DMS-1947438 while in residence at Williams College in Williamstown, MA. 


\section*{Appendix}
\begin{enumerate}
\item \label{2,4timespow5iterateto12} We first prove that $2^{2} \cdot 5^{b}$ iterates to the fixed point $12 \cdot 5^{b}$. Observe: 
    \begin{align}
    z^{2} \left(4 \cdot 5^{b} \right) 
    & \ = \ z \left( z\left( 4 \cdot 5^{b} \right) \right) \nonumber\\
    & \ = \ z \left( \lcm\left( z(4), z(5^{b}) \right) \right) \nonumber\\
    & \ = \ z \left( \lcm\left( 6, 5^{b} \right) \right) \nonumber\\
    & \ = \ z \left( \lcm\left( 6, 5^{b} \right) \right) \nonumber\\
    & \ = \ z \left( 6 \cdot 5^{b} \right) \nonumber\\
    & \ = \ z \left( 6 \cdot 5^{b} \right) \nonumber\\
    & \ = \ \lcm \left( z(2), z(3), z(5^{b}) \right) \nonumber\\
    & \ = \ \lcm \left( 3, 4, 5^{b} \right) \nonumber\\
    & \ = \ 12 \cdot 5^{b}.
    \end{align}
\item Next we prove that $2 \cdot 5^{b}$ iterates to the fixed point $12 \cdot 5^{b}$. Observe:
    \begin{align}
    z^{4} (2 \cdot 5^{b}) 
    & \ = \ z^{3}\left(z(2 \cdot 5^{b}\right) \nonumber\\
    & \ = \ z^{3}\left( \lcm \left( z(2), z(5^{b}) \right) \right) \nonumber\\
    & \ = \ z^{3}\left( \lcm \left( 3, 5^{b} \right) \right) \nonumber\\
    & \ = \ z^{3} \left( 3 \cdot 5^{b} \right) \nonumber\\
    & \ = \ z^{2} \left( z\left( 3 \cdot 5^{b} \right) \right) \nonumber\\
    & \ = \ z^{2} \left( \lcm\left( z(3), z(5^{b}) \right) \right) \nonumber\\
    & \ = \ z^{2} \left( \lcm\left(4, 5^{b} \right) \right) \nonumber\\
    & \ = \ z^{2} \left(4 \cdot 5^{b} \right) \nonumber\\
    & \ = \ 12.
    \end{align}
\end{enumerate}


\bibliographystyle{plain}

\begin{thebibliography}{99}

\bibitem[FM]{FM}
John D. Fulton and William L. Morris. ``On arithmetical functions related to the Fibonacci numbers''. In: \textit{Acta Arithmetica} 2.16 (1969), pp. 105–110.

\bibitem[Ka]{Ka}
Jir{\'i} Kla{\v s}ka. ``Donald Dines Wall’s Conjecture''. In: \textit{Fibonacci Quarterly} 56.1(Feb.2018), pp. 43–51.\href{https://www.fq.math.ca/Papers/56-1/Klaska10917.pdf}{https://www.fq.math.ca/Papers/56-1/Klaska10917.pdf}.


\bibitem[L]{L}
Edouard Lucas. ``Th{\'e}orie des fonctions num{\'e}riques simplement p{\'e}riodiques''. In: \textit{American Journal of Mathematics} (1878), pp. 289–321. doi: \href{https://www.jstor.org/stable/2369373?origin \ = \ crossref}{10.2307/2369373}.\href{https://doi.org/10.2307/2369373}{https://doi.org/10.2307/2369373}.

\bibitem[LT]{LT}
Florian Luca and Emanuele Tron. ``The Distribution of Self-Fibonacci Divisors''. In: Alaca, A., Alaca, s., Williams, K. (eds) \textit{Advances in the Theory of Numbers}. Fields Institute Communications, vol 77. Springer, New York, NY. \href{https://doi.org/10.1007/978-1-4939-3201-6_6}{https://doi.org/10.1007/978-1-4939-3201-6\_6}

\bibitem[M1]{M1}
Diego Marques. ``Fixed points of the order of appearance in the Fibonacci sequence''. In: \textit{Fibonacci Quarterly} 50.4 (Nov. 2012), pp. 346–352. \href{https://www.mathstat.dal.ca/FQ/Papers1/50-4/MarquesFixedPoint.pdf}{https://www.mathstat.dal.ca/FQ/Papers1/50-4/MarquesFixedPoint.pdf}.

\bibitem[M2]{M2}
Diego Marques. ``Sharper upper bounds for the order of appearance in the Fibonacci sequence''. In: \textit{Fibonacci Quarterly} 51.3 (Aug. 2013), pp. 233–238. \href{https://www.fq.math.ca/Papers1/51-3/MarquesSharperUpperBnds.pdf}{https://www.fq.math.ca/Papers1/51-3/MarquesSharperUpperBnds.pdf}.

\bibitem[M3]{M3}
Diego Marques. ``On the order of appearance of integers at most one away from Fibonacci numbers'', In: \textit{Fibonacci Quarterly} 50.1 (2012), 36–43. \href{https://www.fq.math.ca/Papers1/50-1/Marques.pdf}{https://www.fq.math.ca/Papers1/50-1/Marques.pdf}

\bibitem[M4]{M4}
Diego Marques. ``The order of appearance of product of consecutive Fibonacci numbers'', In: \textit{Fibonacci Quarterly} 50.2 (May 2012), 132–139. \href{https://fq.math.ca/Papers1/50-2/MarquesOrderofAppear.pdf}{https://fq.math.ca/Papers1/50-2/MarquesOrderofAppear.pdf}

\bibitem[M5]{M5}
Diego Marques. ``The order of appearance of powers of Fibonacci and Lucas numbers'', In: \textit{Fibonacci Quarterly} 50.3 (Aug. 2012), 239-245.  \href{https://www.fq.math.ca/Papers1/50-3/MarquesPowersFibLucas.pdf}{https://www.fq.math.ca/Papers1/50-3/MarquesPowersFibLucas.pdf}

\bibitem[R]{R}
Marc Stetson Renault. ``Properties of the Fibonacci sequence under various moduli''. MA thesis. Wake Forest University. Department of Mathematics and Computer Science, 1996. \href{http://webspace.ship.edu/msrenault/fibonacci/FibThesis.pdf}{http://webspace.ship.edu/msrenault/fibonacci/FibThesis.pdf}.

\bibitem[S]{S}
H.J.A. Sall{\'e}. ``A Maximum value for the rank of apparition of integers in recursive sequences''. In: \textit{Fibonacci Quarterly} 13.2 (1975), pp. 159–161. \href{https://www.fq.math.ca/Scanned/13-2/salle.pdf}{https://www.fq.math.ca/Scanned/13-2/salle.pdf}.

\bibitem[SM]{SM}
Lawrence Somer and M K{\v{r}}{\'\i}{\v{z}}ek. ``Fixed points and upper bounds for the rank of appearance in Lucas sequences''. In: \textit{Fibonacci Quarterly} 51.4 (Nov. 2013), pp. 291–306. \href{https://www.fq.math.ca/Papers1/51-4/SomerKrizekFixedPoints.pdf}{https://www.fq.math.ca/Papers1/51-4/SomerKrizekFixedPoints.pdf}.

\bibitem[Ta1]{Ta1}
Eva Trojovsk{\'a}. ``On periodic points of the order of appearance in the Fibonacci sequence''. In: \textit{Mathematics} 8.5 (May 2020), p. 773. doi: \href{https://www.mdpi.com/2227-7390/8/5/773}{10.3390/math8050773}. \href{https://doi.org/10.3390/math8050773}{https://doi.org/10.3390/math8050773}.

\bibitem[Ta2]{Ta2}
Eva Trojovsk{\'a}. ``On the Diophantine Equation $ z(n) = \left(2 - 1/k \right)n $ Involving the Order of Appearance in the Fibonacci Sequence''. In: \textit{Mathematics} 8.1 (Jan. 2020), p. 124. doi: \href{https://www.mdpi.com/2227-7390/8/1/124}{10.3390/math8010124}. \href{https://doi.org/10.3390/math8010124}{https://doi.org/10.3390/math8010124}.

\bibitem[Ty]{Ty}
Pavel Trojovsk{\'y}. ``On the Parity of the Order of Appearance in the Fibonacci Sequence''. In: \textit{Mathematics} 9.16 (Aug. 2021), p. 1928. doi: \href{https://www.mdpi.com/2227-7390/9/16/1928#B12-mathematics-09-01928}{10.3390/math9161928}. \href{//www.mdpi.com/2227-7390/9/16/1928}{https://www.mdpi.com/2227-7390/9/16/1928}.


\bibitem[WFM]{WFM}
Wolfram Research Inc. Mathematica, Version 13.3. Champaign, IL, 2023. \href{https://www.wolfram.com/mathematica}{https: //www.wolfram.com/mathematica}.

\end{thebibliography}


\medskip

\noindent MSC2010: {60B10, 11B39  (primary) 65Q30 (secondary)}

\end{document}